\documentclass[11pt,a4paper]{amsart}
\usepackage{amsfonts,amssymb,amscd,amsmath,enumerate,verbatim,calc}
\usepackage{mathrsfs}
\usepackage[all]{xy}
\numberwithin{equation}{section}
\newtheorem{thm}{Theorem}[section]
\newtheorem{cor}[thm]{Corollary}

\newtheorem{prop}[thm]{Proposition}
\newtheorem{defn}[thm]{Definition}

\newtheorem{exam}[thm]{Example}

\DeclareMathOperator{\Ext}{Ext} \DeclareMathOperator{\Supp}{Supp}
\DeclareMathOperator{\V}{V} \DeclareMathOperator{\Hom}{Hom}

 \DeclareMathOperator{\Max}{Max}
\DeclareMathOperator{\lc}{H} 
 \DeclareMathOperator{\Spec}{Spec}
\DeclareMathOperator{\G}{\Gamma} 
\DeclareMathOperator{\E}{E} 

\DeclareMathOperator{\grade}{grade}

\DeclareMathOperator{\Ass}{Ass}

\DeclareMathOperator{\Ht}{ht}
\DeclareMathOperator{\dimSupp}{dimSupp}

\newcommand{\fa}{\mathfrak{a}}
\newcommand{\fb}{\mathfrak{b}}

\newcommand{\fm}{\mathfrak{m}}
\newcommand{\fp}{\mathfrak{p}}

\newcommand{\lo}{\longrightarrow}

\begin{document}

\title[Lower  bounds of general local cohomology modules]
{Lower  bounds of certain general local cohomology modules}

\bibliographystyle{amsplain}

\author[M. Behrouzian]{Mahmoud Behrouzian}
\address{Department of Mathematics, Faculty of Science, Arak University,
	Arak, 38156-8-8349, Iran.}
\email{m-behrooziyan@phd.araku.ac.ir}

   \author[M. Aghapournahr]{Moharram Aghapournahr}
\address{Department of Mathematics, Faculty of Science, Arak University,
Arak, 38156-8-8349, Iran.}
\email{m-aghapour@araku.ac.ir}






\begin{abstract}
Let $R$ be a commutative Noetherian ring, $\Phi$ a system of ideals of $R$, $\fa \in \Phi$,
$M$  an arbitrary $R$-module and $t$ a non-negative integer. Let  $\mathcal{S}$ be a Melkersson subcategory of $R$-modules. Among other things, we prove that if $\lc^{i}_\Phi(M)$ is in $\mathcal{S}$ for all $i < t$ then  $\lc^{i}_\fa(M)$ is in $\mathcal{S}$ for all $i < t$ and  for all $\fa \in \Phi$. If $\mathcal{S}$ is the class of $R$-modules $N$ with $\dim N \leq k$
where $k \geq -1$, is an integer, then $\lc^{i}_\Phi(M)$ is in $\mathcal{S}$ for all $i < t$ (if and only if) $\lc^{i}_\fa(M)$ is in $\mathcal{S}$ for all $i < t$ and for all  $\fa \in \Phi$. As consequences we study and compare vanishing, Artinianness and support of general local cohomology and ordinary local cohomology supported at ideals of its system of ideals at initial points $i <t$. We show that $\Supp_{R}(\lc^{\dim M-1}_{\Phi}(M))$ is not necessarily finite whenever $(R,\fm)$ is local and $M$ a finitely generated $R$-module.  
\end{abstract}

\maketitle


\section{Introduction}

Throughout  this paper $R$ is a commutative Noetherian ring with non-zero
identity and $\fa$ an ideal of $R$. For an $R$-module $M$,  the $i^{th}$
local cohomology module $M$ with respect to ideal $\fa$ is defined as
\begin{center}
$\lc^{i}_{\fa}(M) \cong \underset{n}\varinjlim \Ext^{i}_{R}(R/{\fa}^{n},M).$
\end{center}





There are some generalizations of the theory of ordinary local cohomology modules.
The following is introduced by Bijan-Zadeh in \cite{BZ1}.

Let $\Phi$ be a non-empty set of ideals of $R$. We call $\Phi$ a {\it system of
	ideals} of $R$ if, whenever $\fa_1,\fa_2\in\Phi $, then there is an ideal $\fb\in\Phi$
such that $\fb\subseteq \fa_1{\fa_2}$. For such a system, for every $R$-module $M$,
one can define
$$\G_{\Phi}(M)=\{~x\in M\mid {\fa}x=0 \textmd{\ \ for some}~ \fa\in \Phi\}.$$

Then $\G_{\Phi}(-)$ is a functor from $\mathscr{C}(R)$ to itself (where
$\mathscr{C}(R)$ denotes the category of all $R$-modules and all
$R$-homomorphisms). The functor  $\G_{\Phi}(-)$ is additive, covariant,
$R$-linear and left exact. In \cite{BZ2}, $\G_{\Phi}(-)$ is denoted by $L_{\Phi}(-)$
and is called the \textquotedblleft general local cohomology functor with respect to $\Phi$\textquotedblright. For
each $i\geq 0$, the $i$-th right derived functor of $\G_{\Phi}(-)$ is denoted by
$\lc_{\Phi}^i(-)$. The functor $\lc_{\Phi}^i(-)$ and $\underset{\fa\in
	\Phi}\varinjlim\lc_{\fa}^i(-)$ (from $\mathscr{C}(R)$ to itself) are naturally equivalent
(see \cite{BZ1}). For an ideal $\fa$ of $R$, if $\Phi=\{{\fa}^n| n\in \mathbb{N}_0\}$, then
the functor $\lc_{\Phi}^i(-)$ coincides with the ordinary local cohomology functor
$\lc_{\fa}^i(-)$.
It is shown that, the study of torsion theory over $R$ is equivalent to study the
	general local cohomology theory (see \cite{BZ2}).
It is easy to see that the above definition of system of ideals and general local cohomology modules is equivalent to the \cite[Definition 2.1.10 and Notation 2.2.2]{BSh2}. General local cohomology modules  was studied  by several authors in  \cite{BZ1, BZ2, DNT, AB1, A2, A3, BSh2}.


Let $\mathcal{S}$ be a Serre subcategory of the category of all $R$-modules. The crucial points of Section 2, are Theorems \ref{2.1} and \ref{2.8} which show when the $R$-modules $\Ext_{R}^{s+t}(N , M)$
and $\Ext_{R}^{s}(N , \lc_{\Phi}^{t}(M))$ belong to $\mathcal{S}$, where $s,t$ are non-negative integers. These two theorems, which
are frequently used through the paper, enable us to demonstrate some new facts and improve some older facts about the extension functors of general local cohomology modules. Our main result in this section are Corollaries \ref{2.6}, \ref{mel}  and \ref{2.14} as follows:

\begin{cor}\textsl{{\rm(}See Corollaries \ref{2.6} and \ref{mel}{\rm)}}
	 Let $M$ be an arbitrary $R$-module,  $\mathcal{S}$ a Melkersson subcategory with respect to the any ideal $\fa \in \Phi$ and $N$ be a finitely generated   $R$-module  with $\Supp_{R}N= V(\fa)$  for some  $\fa \in \Phi $. Let $t$ be a non-negative integer such that $\Ext^{j}_R(N,\lc^{i}_\Phi(M))$ is in $\mathcal{S}$ for all   $i <  t $ and all  $j < t-i$, then $\lc^{i}_\fa(M) \in \mathcal{S}$ for all  $i < t$. In particular, if $\lc^{i}_\Phi(M)$ is in $\mathcal{S}$ for all   $i < t $, then $\lc^{i}_\fa(M) \in \mathcal{S}$ for all  $i < t$.  
\end{cor}

\begin{cor}\textsl{{\rm(}See Corollary \ref{2.14}{\rm)}}
	Let  $\mathcal{S}$  be the class of all $R$-modules $N$ with $ \dim_{R}N \leq k $, where  $k$ is an integer and $M$ an arbitrary $R$-module. Then the following statements are equivalent:
	\begin{itemize}
		\item[(i)]  $\lc^{i}_\Phi(M) \in \mathcal{S}$ for all $i<t$ {\rm(}for all $i \geq 0${\rm)};
		\item[(ii)]  $\lc^{i}_\fa(M) \in \mathcal{S}$ for all $i<t$ {\rm(}for all $i \geq 0${\rm)}  and all  $\fa \in \Phi$.
	\end{itemize}
\end{cor}

In section 3, first of all we generalize and improve \cite[Theorem 2.2]{T} (Local-global Principle for the Artinianness of local cohomology modules). Then we prove the following main result:

 \begin{cor}\textsl{{\rm(}See Corollary \ref{3.4}{\rm)}}\label{1.3}
 	Let  $M$  be an $\fa$-$ETH$-cofinite  $R$-module for all ideal $\fa \in \Phi$ or finitely generated and $t$ a positive integer.  Then the following conditions are equivalent:
 	\begin{itemize}
 		\item[(i)]  $\Supp(\lc^{i}_{\Phi}(M))  \subseteq \Max(R)$ for all $i<t$ {\rm(}for all $i\geq0${\rm)};
 		\item[(ii)]  $\Supp(\lc^{i}_{\fa}(M))  \subseteq \Max(R)$  for all $i<t$  {\rm(}for all $i\geq0${\rm)} and all $\fa \in \Phi$;
 		\item[(iii)]  $\lc^{i}_\fa(M) $  is Artinian $R$-module for all $i<t$ {\rm(}for all $i\geq0${\rm)} and all  $\fa \in \Phi$;
 	\end{itemize}			
 	when $R$ is a semi-local ring these conditions are also equivalent to: 
 	\begin{itemize}	 
 		\item[(iv)]  $\lc^{i}_\Phi(M) $  is Artinian for all $i<t$ {\rm(}for all $i\geq0${\rm)}.
 	\end{itemize}
 \end{cor}

As an application of Corollary \ref{1.3}, we prove the following theorem:

\begin{thm}\textsl{{\rm(}See Theorem \ref{cof}{\rm)}}
	Let $R$ be a semi-local Noetherian ring, $\Phi$ a system of ideals of $R$ and $M$  a finitely generated $R$-module. If $\dim R/{\fa}\leq0$ for all $\fa\in \Phi$, then $\lc^{i}_{\Phi}(M)$ is $\Phi$-cofinite Artinian for all $i\ge 0.$
\end{thm}

In Example \ref{ex}, we show that implication $(iii) \Leftrightarrow (iv)$ of Corollary \ref{1.3} is not true in general. In Corollary \ref{cor2.15}, we show that the vanishing of $\lc^{i}_{\Phi}(M)$ for all $i < t$ (for all $i\geq 0$) is equivalent to the vanishing of $\lc^{i}_{\fa}(M)$
for all $i < t$ (for all $i\geq 0 $) and all $\fa \in \Phi$. Using this vanishing result in Corollary \ref{2.16}, we prove the following equality: 
	\begin{center}
		$\underset{i<t}\bigcup \Supp_R(\lc^{i}_{\Phi}(M))= \underset{\fa\in\Phi,i<t}\bigcup \Supp_R(\lc^{i}_{\fa}(M)).$
	\end{center}

In Section 4, we get some applications of theorems \ref{2.1} and \ref{2.8}. Then for non-negative integers $s$ and $t$, we find some sufficient conditions
for validity of the isomorphism:

\begin{center}
	$\Ext^{s}_R(N,\lc^{t}_\Phi(M)) \cong \Ext^{s+t}_R(N,M) \cong \lc^{s+t}_\fa(N,M)  \cong \Ext^{s}_R(N,\lc^{t}_\fa(M))$
\end{center}
As a consequence, we prove the following corollary:

\begin{cor}\textsl{{\rm(}See Corollary \ref{4.9}{\rm)}}
	Let $M$, $N$ be two finitely generated   $R$-modules such that $\Supp_R(N) = \V(\fa)$  for some $\fa \in \Phi $.  If $\grade(\Phi,M)=t$, then 
	\begin{center}
		$\Hom_R(N,\lc^{t}_\Phi(M)) \cong  \Ext^{t}_R(N,M) \cong  \lc^{t}_\fa(N,M) \cong \Hom_R(N,\lc^{t}_\fa(M)).$
	\end{center} 
\end{cor}

We close the paper with the example \ref{4.8} that shows, for a local ring $(R,\fm)$ and a finitely generated $R$-module
$M$, with $\dim M = d$, the set $\Supp_{R}(\lc^{d-1}_{\Phi}(M))$ is not necessarily finite.

Even though we can show some of our results by using \textbf{spectral sequences}, we are avoiding the use of this technique completely in this work and we provide more \textbf{elementary proofs} for the results.

Throughout this paper, $R$ will always be a commutative Noetherian ring with
non-zero identity, $\Phi$ a system of ideals of $R$  and $\fa$ will be an ideal of $R$.  We denote $\{\fp \in {\rm
	Spec}\,R:\, \fp\supseteq \fa \}$ by $V(\fa)$. As a convention we consider $\dim 0= -1$. 
For any unexplained notation and terminology we refer the reader to \cite{BSh2}, \cite{Mat} and \cite{BH}.




\section{Lower bounds of general local cohomology}


 Recall that a \textit{Serre subcategory} $\mathcal{S}$ of the category of $R$--modules is a subclass of $R$--modules such that for any short exact sequence\\
\centerline{$0\longrightarrow X'\longrightarrow X\longrightarrow X''\longrightarrow 0,$}\\
the module $X$ is in $\mathcal{S}$ if and only if $X'$ and $X''$ are in $\mathcal{S}$. In other words, it is closed under taking submodules, quotients and extensions.

We begin the paper with a useful theorem that is a generalization of \cite[Theorem 2.1]{ATV}.

\begin{thm}\label{2.1}
 Let $M$ be an arbitrary $R$-module,  $\mathcal{S}$ a Serre subcategory and $N$ be a $\Phi$-torsion $R$-module. If $\Ext^{t-i}_R(N,\lc^{i}_\Phi(M)) \in \mathcal{S}$ for all   $i \leq t$. Then  
$\Ext^{t}_R(N,M) \in \mathcal{S}$.

\end{thm}
\begin{proof}
We prove the result by induction on  $t$. The case $t=0$ is clear, because  if $t=0$ then,  $i=0$ so
$$\Ext^{0}_R(N,\lc^{0}_\Phi(M)) = \Hom_R(N,\Gamma_\Phi(M)) \cong \Hom_R(N,M). $$ 
Suppose that $t > 0$ and $t-1$ is settled. Let $\overline{M} = M / \Gamma_\Phi(M)$, $L = E(\overline{M}) /\overline{M}$  where $E(\overline{M})$  is the injective hull of $\overline{M}$. Since 
$\Gamma_\Phi(\overline{M}) = 0 = \Gamma_\Phi(E(\overline{M}))  $, by applying the derived functors $ \Gamma_\Phi(-)$ and $\Hom_R(N,-)$  on the short exact sequence
$$0 \longrightarrow \overline{M} \longrightarrow E(\overline{M}) \longrightarrow L \longrightarrow 0$$
we obtain for all  $i > 0$, the isomorphisms,
$$\lc^{i-1}_\Phi(L) \cong \lc^{i}_\Phi(\overline{M}) (\cong  \lc^{i}_\Phi(M)) ~~, ~~ \Ext^{i-1}_R(N,L) \cong \Ext^{i}_R(N,\overline{M}).$$
From the above isomorphisms for all $0 \leq i \leq t-1$ we have,
 $$\Ext^{(t-1)-i}_R(N,\lc^{i}_\Phi(L)) \cong \Ext^{t-(i+1)}_R(N,\lc^{i+1}_\Phi(M))$$
 which is in $\mathcal{S}$ by assumptions. Thus, from the induction hypothesis  on $L$, $\Ext^{t-1}_R(N,L) \in \mathcal{S}$.
 \\
 Now by the short exact sequence $ 0 \longrightarrow \Gamma_\Phi(M) \longrightarrow M \longrightarrow \overline{M} \longrightarrow 0 $  we get the long exact sequence
  $$\cdots \longrightarrow \Ext^{t}_{R} (N , \Gamma_\Phi(M)) \longrightarrow \Ext^{t}_{R} (N , M) \longrightarrow 
\Ext^{t}_{R} (N , \overline{M}) \longrightarrow \cdots  $$
which shows that $\Ext^{t}_R(N,M) \in \mathcal{S}$.
\end{proof}

\begin{cor}\label{2.2}
Let $M$ be an arbitrary $R$-module,  $\mathcal{S}$ a Serre subcategory and $N$ be a $\Phi$-torsion $R$-module. Let $t$ be a non-negative integer  such that $\Ext^{j}_R(N,\lc^{i}_\Phi(M)) \in \mathcal{S}$ for all $i < t$ and all $j < t-i$ Then  
$\Ext^{i}_R(N,M) \in \mathcal{S}$   for all  $i < t$.
\end{cor}
\begin{cor}\label{2.3}
Let $M$ be an arbitrary $R$-module,  $\mathcal{S}$ a Serre subcategory and $N$ be a finite $\fa$-torsion $R$-module for some $\fa \in \Phi$. If $\Ext^{t-i}_R(N,\lc^{i}_\Phi(M))$ is in $\mathcal{S}$ for all   $i \leq t$, then $\lc^{t}_\fa(N,M) \in \mathcal{S}$.  
\end{cor}
\begin{proof}
The result follows by Theorem \ref{2.1} and \cite [theorem 2.5]{VA}. Note that $N = \Gamma_\fa (N) \subseteq \Gamma_\Phi(N)$. 
\end{proof}


\begin{cor}\label{2.4}
Let $M$ be an arbitrary $R$-module,  $\mathcal{S}$ a Serre subcategory and $N$ be a finite $\fa$-torsion $R$-module for some $\fa \in \Phi$. If $\Ext^{j}_R(N,\lc^{i}_\Phi(M))$ is in $\mathcal{S}$ for all   $i < t$    and all     $j <  t-i$, then  
$\lc^{i}_{\fa}(N,M) \in \mathcal{S}$   for all  $i < t$.
\end{cor}


The second author of present paper and Melkersson in \cite[Definition 2.1]{AM2008} defined Melkersson subcategory with respect to an ideal as follows:

\begin{defn}\label{def}
	A full subcategory $\mathcal{S}$ of the category of  $R$-modules is said to be Melkersson  subcategory  with respect to the ideal $\fa$   if for any $\fa$-torsion  $R$-module $M$, $(0 :_{M} \fa)\in \mathcal{S}$ implies $M \in \mathcal{S}$.                                                
\end{defn}

 To see some examples of Melkersson subcategories, we refer
the reader to \cite[Examples 2.4 and 2.5]{AM2008}.


The following two corollaries are our first main results of this paper.

\begin{cor}\label{2.6}
 Let $M$ be an arbitrary $R$-module,  $\mathcal{S}$ a Melkersson subcategory with respect to the any ideal $\fa \in \Phi$ and $N$ be a finitely generated   $R$-module  with $\Supp_{R}N= V(\fa)$  for some  $\fa \in \Phi $. Let $t$ be a non-negative integer such that $\Ext^{j}_R(N,\lc^{i}_\Phi(M))$ is in $\mathcal{S}$ for all $i < t$ and all  $j < t-i$, then $\lc^{i}_\fa(M) \in \mathcal{S}$ for all  $i < t$. 
\end{cor}

\begin{proof}
It follows by Corollary
\ref{2.2} and  \cite[Theorem 2.9]{AM2008}.	
\end{proof}


\begin{cor}\label{mel}
	Let $M$ be an arbitrary $R$-module,  $\mathcal{S}$ a Melkersson subcategory with respect to the any ideal $\fa \in \Phi$ and $t$ a non-negative integer. If $\lc^{i}_\Phi(M)$ is in $\mathcal{S}$ for all   $i < t $, then $\lc^{i}_\fa(M) \in \mathcal{S}$ for all  $i < t$. 
\end{cor}



\begin{cor}\label{2.7}
Let $M$ be an arbitrary $R$-module and  $\mathcal{S}$ a Melkersson subcategory with respect to the any ideal $\fa \in \Phi$ then:
\begin{center}
$\inf  \lbrace i : \lc^{i}_\Phi(M) \notin \mathcal{S} \rbrace \leq \inf \lbrace \inf \lbrace i : \lc^{i}_\fa(M) \notin \mathcal{S} \rbrace :  \fa \in \Phi \rbrace.$ 
\end{center}
  \end{cor}

The following theorem is a generalization of \cite[Theorem 2.3]{ATV}.

\begin{thm}\label{2.8}
Let $N$ be a $\Phi$-torsion $R$-module, $M$  be an arbitrary $R$-module, $\mathcal{S}$ a Serre subcategory of $R$-modules and $s,t$ be non-negative integers such that
\begin{itemize}
\item[(i)] $\Ext^{s+t}_R(N,M)$  is in $\mathcal{S}$; 
\item[(ii)] $\Ext^{s+1+i}_R(N,\lc^{t-i}_\Phi(M))$  is in $\mathcal{S}$ for all $i$, $1 \leq i \leq t$; 
\item[(iii)] $\Ext^{s-1-i}_R(N,\lc^{t+i}_\Phi(M))$  is in $\mathcal{S}$  for all $i$, $1 \leq i \leq s-1$. 
\end{itemize}
Then $\Ext^{s}_R(N,\lc^{t}_\Phi(M)) \in \mathcal{S}$.
\end{thm}

\begin{proof}
We prove by induction on $t$. Let $t=0$ and set $\overline{M} = M / \Gamma_{\Phi}(M)$. By hypothesis (iii)    
$\Ext^{(s-1)-i}_R(N,\lc^{i}_\Phi(\overline{M}))$  is in $\mathcal{S}$ for all $i$, $0 \leq i \leq s-1$. Thus from Theorem \ref{2.1}, 
$\Ext^{s-1}_R(N,\overline{M})$ belongs to $\mathcal{S}$ . Applying  the derived functor $\Hom_{R}(N , -)$ to the short exact sequence 
$$0 \longrightarrow \Gamma_{\Phi}(M)  \longrightarrow M \longrightarrow \overline{M} \longrightarrow 0$$ we obtain the long exact sequence 
$$\cdots \longrightarrow \Ext^{s-1}_R(N,\overline{M}) \longrightarrow \Ext^{s}_R(N,\Gamma_{\Phi}{M}) \longrightarrow \Ext^{s}_R(N,M) \longrightarrow \cdots$$
Since  $\Ext^{s-1}_R(N,\overline{M})$ is in $\mathcal{S}$  and by  $t=0$ in (i), $\Ext^{s}_R(N,M)$ is in $\mathcal{S}$ which   shows that  $\Ext^{s}_R(N,\Gamma_{\Phi}{M})  \in \mathcal{S}$. Now assume that $t > 0$   and that $t-1$ is settled. Let    $\overline{M} = M / \Gamma_{\Phi}(M)$, $L = E(\overline{M}) / \overline{M}$  where $E(\overline{M})$  is the injective hull of $\overline{M}$.
By applying the derived functors $ \Gamma_\Phi(-)$ and $\Hom_R(N,-)$  on the short exact sequence
$$0 \longrightarrow \overline{M} \longrightarrow E(\overline{M}) \longrightarrow L \longrightarrow 0$$ the proof sufficiently similar to that of Theorem \ref{2.1} to be omitted. We leave the proof to the reader.
\end{proof}




The following corollary is a generalization of \cite[Corollary 2.4]{ATV}.

\begin{cor}\label{2.11}
Let $M$ an arbitrary $R$-module and  $N$  a $\Phi$-torsion $R$-module. Let $\mathcal{S}$ be a Serre subcategory of $R$-modules and $t$ a non-negative integer. If $\Ext^{t+1-i}_R(N,H^{i}_\Phi(M))\in \mathcal{S}$  for all   $i < t$ and $\Ext^{t}_R(N,M) \in \mathcal{S}$    then  $\Hom_R(N,H^{t}_\Phi(M)) \in \mathcal{S}$. 
\end{cor}

\begin{proof}
Put $s=0$   in Theorem \ref{2.8}.
\end{proof}
\begin{prop}\label{2.12}
Let  $\mathcal{S}$   be the class of all  $R$-modules       $N$ with $ \dim_{R}N \leq k $,    where  $k$    is an integer and $M$  an arbitrary $R$-module. Let $t$ be a non-negative integer such that   
 $\Ext^{t+1-i}_R(R /\fa,\lc^{i}_\Phi(M)) \in \mathcal{S}$  for all    $i < t$ and all $\fa \in \Phi$ and  $\Ext^{t}_R(R / \fa, M) \in \mathcal{S}$ for all $\fa \in \Phi$, then  $\lc^{t}_\Phi(M) \in \mathcal{S}$.
\end{prop} 
\begin{proof}
We have $\Hom_R(R /\fa,\lc^{t}_\Phi(M)) \in \mathcal{S}$   for all  $\fa \in \Phi$  by  Corollary \ref{2.11} .    Also, hence $\lc^{t}_\Phi(M)$ is a $\Phi$-torsion $R$-modules, thus
$$\lc^{t}_\Phi(M) = \bigcup _{\fa \in \Phi} (0 :_{\lc^{t}_\Phi(M)} \fa) = \bigcup _{\fa \in \Phi} \Hom_R(R / \fa,\lc^{t}_\Phi(M)). $$
Thus,  $\Supp_R \lc^{t}_\Phi(M)\subseteq \bigcup _{\fa \in \Phi} \Supp_R\Hom_R(R /\fa,\lc^{t}_\Phi(M))$, and therefore $\lc^{t}_\Phi(M) \in \mathcal{S} $.
\end{proof}

\begin{cor}
Let  $\mathcal{S}$   be the class of all  $R$-modules       $N$ with $ \dim_{R}N \leq k $,    where  $k$    is an integer and $M$  an arbitrary $R$-module. Let    
 $\Ext^{t}_R(R /\fa,M)  \in \mathcal{S}$  for all    $i < t$ and all $\fa \in \Phi$.     Then  $\lc^{i}_\Phi(M) \in \mathcal{S}$    for all $i < t$.
\end{cor} 
\begin{proof}
By a similar proof   of Proposition \ref{2.12}, we can show that $\Gamma_{\Phi}(M) \in \mathcal{S}$. Now it follows 
from Proposition \ref{2.12} that $\lc_{\Phi}^{1}(M) \in \mathcal{S}$. By keeping this process we have  
$\lc_{\Phi}^{i}(M) \in \mathcal{S}$ for all $i < t$.
\end{proof}
\begin{cor}\label{2.13}
Let $M$ be an arbitrary $R$-module. If $k\geq -1$ be an integer such that  $ \dim_{R}\lc^{i}_\fa(M) \leq k$   for all   $i < t$ {\rm(}resp. for all $i${\rm)}  and all  $\fa \in \Phi$,  then  $ \dim_{R}\lc^{i}_\Phi(M) \leq k$   for all  $i < t$ {\rm(}resp. for all $i${\rm)}. 
\end{cor}
\begin{proof}
The result follows by Corollary \ref{2.13} and \cite[Theorem 2.9]{AM2008}. 
\end{proof}

The following result is one of the main results of this paper.

\begin{cor}\label{2.14}
Let  $\mathcal{S}$  be the class of all $R$-modules $N$ with $ \dim_{R}N \leq k $, where  $k$ is an integer, and $M$ an arbitrary $R$-module. Then the following statements are equivalent:
\begin{itemize}
		\item[(i)]  $\lc^{i}_\Phi(M) \in \mathcal{S}$ for all $i<t$ {\rm(}for all $i \geq 0${\rm)};
		\item[(ii)]  $\lc^{i}_\fa(M) \in \mathcal{S}$ for all $i<t$ {\rm(}for all $i \geq 0${\rm)}  and all  $\fa \in \Phi$.
	\end{itemize}
\end{cor}

\begin{proof}
It follows by Corollaries \ref{2.13} and \ref{mel}.	
\end{proof}




\begin{cor}\label{2.15}
Let  $\mathcal{S}$  be the class of all $R$-modules $N$ with $ \dim_{R}N \leq k $, where  $k$ is an integer, and $M$ an arbitrary $R$-module. Then:
\begin{align*}
inf \lbrace  i : \lc^{i}_\Phi(M) \notin \mathcal{S} \rbrace & =inf \lbrace  inf \lbrace i : \lc^{i}_\fa(M) \notin \mathcal{S} \rbrace :  \fa \in \Phi \rbrace\\
                                        & =inf \lbrace  inf \lbrace i : \Ext^{i}_R(R /\fa,M)\notin \mathcal{S} \rbrace :  \fa \in \Phi \rbrace.                   
\end{align*}
\end{cor} 
\begin{proof}
By Corollary \ref{2.14}, \cite[Theorem 2.9]{AM2008}.
\end{proof}

\section{Cofiniteness, Artinianness and vanishing}

 As a generalization of definition of cofinite modules with respect to an ideal (\cite[Definition
 2.1]{MEL 2005}), it is introduced in \cite[Definition 2.2]{A} the following definition.
 
 \begin{defn}
 	Let $\fa$ be an ideal of $R$, then the  $R$-module $M$ {\rm (}not necessary $\fa$-torsion{{\rm )}} is called $ETH$-cofinite with respect to $\fa$  or $\fa$-$ETH$-cofinite, if $\Ext^{i}_R(R /\fa,M)$ is a finitely generated $R$-module for all $i$. 
 \end{defn}
 
 To see the properties of this class see \cite[Remark 2.2 and Example 2.3]{AB1}.

 The following theorem which is Local-global Principle for the Artinianness of local cohomology modules generalizes and improves 
 \cite[Theorem 2.2]{T}. Tang in \cite[Theorem 2.2]{T} proved the following theorem when $M$ is finitely generated, whereas here, with a completely different proof,  we prove it for the class of $\fa$-$ETH$-cofinite modules. The implication  $(iv)\Rightarrow(ii)$ is also a generalization of \cite[Proposition 2.2 and Corokkary 2.3]{BN}.
 
 
 \begin{thm}\label{th3.3}
 	\textsl{\rm(Local-global Principle for the Artinianness of local cohomology modules).} Let $\fa$ be a proper ideal of $R$, $M$ an  $\fa$-ETH-cofinite  $R$-module and $t$ a positive integer. Then the following statements are equivalent:
 	\begin{itemize}
 		\item[(i)] $\lc^{i}_{\fa}(M)$ is Artinian for all $i<t$ {\rm(}for all $i${\rm)};
 		\item[(ii)] $\lc^{i}_{\fa}(M)$ is Artinian and $\fa$-cofinite for all $i<t$ {\rm(}for all $i${\rm)};
 		\item[(iii)] $\lc^{i}_{\fa}(M)_{\fp}$ is Artinian for all $i<t$ {\rm(}for all $i${\rm)   and all prime ideal $\fp\in \Spec(R)$};
 		\item[(iv)] $\Supp(\lc^{i}_{\fa}(M))  \subseteq \Max(R)$ for all $i<t$ {\rm(}for all $i${\rm)}.
 	\end{itemize}
 	 \end{thm}
 	\begin{proof}
 		The implications $(i) \Rightarrow (iii)$, $(ii) \Rightarrow (i)$ and  $(ii) \Rightarrow (iv)$ are obviously true. 
 		
 		To prove implication  $(iii)  \Rightarrow (ii)$, we use induction on $t$.  Let $t=1$. Since
 		\begin{center} 
 		$\Hom_{R}(R/\fa , \Gamma_{\fa} (M))= \Hom_{R}(R/\fa , M),$
 		\end{center}
 		 it follows easily that the set $\Ass_R(\G_{\fa}(M))$ is finite. Let $\Ass_R(\G_{\fa}(M))=\{\fp_1,\dots, \fp_n \}$. So we have 
 		
 		\begin{center}
 			$\E(\G_{\fa}(M))=\oplus_{i=1}^{n}\mu^0({\fp}_i,\G_{\fa}(M))\E(R/{\fp}_i)$.
 		\end{center} 
 		
 		By localizing of the above equality at each $\fp_i$ of $\Ass_R(\G_{\fa}(M))$ we get  
 			$\Ass_R(\G_{\fa}(M)) \subseteq \Max(R)$
 		
 		and $\mu^0({\fp}_i,\G_{\fa}(M))$ is finite. So it follows  $\E(\G_{\fa}(M))$ and therefore $\G_{\fa}(M)$ is an  Artinian $R$-module. Hence $\Hom_{R}(R/\fa , \Gamma_{\fa} (M))$ has finite lengths.  Now, it follows by  \cite[Proposition 4.1]{MEL 2005}, $\Gamma_{\fa} (M)$ is Artinian and $\fa$-cofinite 
 		$R$-module. Suppose $t>1$ and the case $t-1$ is settled. Since by induction hypothesis  $\lc^{i}_{\fa}(M)$ is Artinian and $\fa$-cofinite for all $i<t-1$. It follows by \cite[Corollary 4.4]{AM} that $\Hom_{R}(R/\fa , \lc^{t-1}_{\fa}(M))$ is a finitely generated $R$-module. Therefore the set $\Ass_R(\lc^{t-1}_{\fa}(M))$ is finite. The rest of proof is the same as the case $t=1$. Therefor  $\lc^{t-1}_{\fa}(M)$ is an Artinian and $\fa$-cofinite $R$-module, this complete the inductive step.
 		
 		To prove the implication  $(iv) \Rightarrow (ii)$ again we use  induction on $t$. Let $t=1$. Since $\Supp_R( \Gamma_{\fa} (M)) \subseteq \Max(R)$, thus
 		 	$\Supp_{R}(\Hom_{R}(R/\fa , \Gamma_{\fa} (M))) \subseteq \Max(R)$. On the other hand since 
 		$ \Hom_{R}(R/\fa , \Gamma_{\fa} (M))  \cong \Hom_{R}(R/\fa , M)$, so   
 		$\Hom_{R}(R/\fa , \Gamma_{\fa} (M))$ is finitely generated $R$-module. It is easy to see that  $\Hom_{R}(R/\fa , \Gamma_{\fa} (M))$ is an Artinian $R$-module and therefore $\Hom_{R}(R/\fa , \Gamma_{\fa} (M))$ has finite lengths. It follows by  \cite[Proposition 4.1] {MEL 2005} that $\Gamma_{\fa} (M)$ is Artinian and $\fa$-cofinite 
 		$R$-module.
 		Now, assume that $t > 1$ and the case $t-1$ is settled. Consider $\overline{M}= M/\Gamma_{\fa}(M)$. Since 
 		$\Gamma_{\fa}(M) =  \lc^{0}_{\fa}(M)$   is Artinian and $\fa$-cofinite $R$-module, it is easy to see that $M$ is 
 		$\fa$-$ETH$-cofinite if and only if  $\overline{M}$ is  $\fa$-$ETH$-cofinite. Further, note that   $ \lc^{i}_{\fa}(\overline{M}) \cong \lc^{i}_{\fa}(M), ~~ i>0$. Thus we may assume $M$ is $\fa$-torsion-free and so   $\G_{\fa}(M) = 0$. Let $E$ be an injective hall of $M$ and put $L=E/M$. Then also $\Gamma_{\fa}(E)=0$. Therefore the exact sequence $0 \longrightarrow M \longrightarrow E \longrightarrow L \longrightarrow 0,$	implies $\lc^{i}_{\fa}(L)\cong \lc^{i+1}_{\fa}(M)$ and
 		$\Ext^{i}_R(R /\fa,L)\cong \Ext^{i+1}_R(R /\fa,M)$ for all $i\geq 0$. It follows that $L$ is an $\fa$-$ETH$-cofinite $R$-module and $\Supp_{R}(\lc^{i}_{\fa}(L)) \subseteq \Max(R)$ for all $i < t-1$. Therefore by induction hypothesis  $\lc^{i}_{\fa}(L)$ is Artinian and $\fa$-cofinite $R$-module for all $i < t-1$. The isomorphism $\lc^{i}_{\fa}(L)\cong \lc^{i+1}_{\fa}(M)$ shows that $\lc^{i}_{\fa}(M)$ is Artinian and $\fa$-cofinite $R$-module for all $i < t$. This complete the inductive step.
 		\end{proof}


\begin{cor}\label{sez3.9}
Let $\fb\subseteq\fa$ be two ideals of a Noetherian ring $R$, $M$ be a $\fb$-cofinite $R$-module and $t$ be a positive integer. Then the above theorem holds.
\end{cor}

\begin{proof}
	Note that $\Supp_R(R/{\fa})=\V(\fa)\subseteq\V(\fb)$,  therefore ${\Ext}^{i}_R(R/\fa,M)$ is finitely generated for all $i\geq 0$ by \cite[Corollary 2.5]{MEL 2005}. Now the assertion follows by Theorem \ref{th3.3}.
\end{proof}

 			

 
The following corollary that is one of our main results, shows that the Artinianness of general local 
cohomology  at initial points is  not necessarily  equivalent to the Artinianness  of its ordinary local cohomology in general and it needs some more conditions.

 \begin{cor}\label{3.4}
 	Let  $M$  be an $\fa$-$ETH$-cofinite  $R$-module for all ideal $\fa \in \Phi$ or finitely generated and $t$ a positive integer.  Then the following statements are equivalent:
 	\begin{itemize}
 		\item[(i)]  $\Supp(\lc^{i}_{\Phi}(M))  \subseteq \Max(R)$ for all $i<t$ {\rm(}for all $i\geq0${\rm)};
 		\item[(ii)]  $\Supp(\lc^{i}_{\fa}(M))  \subseteq \Max(R)$  for all $i<t$  {\rm(}for all $i\geq0${\rm)} and all $\fa \in \Phi$;
 		\item[(iii)]  $\lc^{i}_\fa(M) $  is Artinian $R$-module for all $i<t$ {\rm(}for all $i\geq0${\rm)} and all  $\fa \in \Phi$; 
 	\end{itemize}			
 	when $R$ is a semi-local ring these conditions are also equivalent to: 
 	\begin{itemize}	 
 	\item[(iv)]  $\lc^{i}_\Phi(M) $  is Artinian for all $i<t$ {\rm(}for all $i\geq0${\rm)}.
 	\end{itemize}
 \end{cor}
 \begin{proof}
 	$(i) \Leftrightarrow (ii)$  follows by Corollary \ref{2.14}.\\
 	$(ii) \Leftrightarrow (iii)$  follows by Theorem \ref{th3.3}.\\
 	$(iii) \Leftrightarrow (iv)$  follows by \cite[Lemma 3.2]{DFT}, \cite[Lemma 3.2]{DNT} and \cite[Theorem 2.9 $(i) \Leftrightarrow (ii)$]{AM2008}.
 \end{proof}
 
 
The following definition is a generalization of the concept $\fa$-cofinite modules which was introduced by Hartshorne.
\begin{defn}[\cite{AKS}]
	Let $\Phi$ be a system of ideals of $R$ and $X$ an $R$-module. The general local cohomology
	module $\lc^j_\Phi(X)$ is called to be $\Phi$-cofinite if there exists an ideal $\fa \in \Phi$ such that
	$\Ext^i_R(R/\fa, \lc^j_\Phi(X))$ is finitely generated for all $i.$
\end{defn}

\begin{thm}\label{cof}
Let $R$ be a semi-local Noetherian ring, $\Phi$ a system of ideals of $R$ and $M$  a finitely generated $R$-module. If $\dim R/{\fa}\leq0$ for all $\fa\in \Phi$, then $\lc^{i}_{\Phi}(M)$ is $\Phi$-cofinite Artinian for all $i\ge 0.$
\end{thm}

\begin{proof}
	Since  by \cite[Lemma 2.1]{BZ1}, $$\lc^{i}_\Phi(M)\cong\underset { \fa \in \Phi}
	\varinjlim \lc^{i}_\fa(M),$$ it is easy to see that $\Supp_R(\lc^{i}_\Phi(M))
	\subseteq\underset{\fa \in \Phi}\bigcup \Supp_R(\lc^{i}_\fa(M))$ and therefore
	$$\dimSupp \lc^{i}_\Phi(M) \leq \sup \{ \dimSupp \lc^{i}_{\fa}(M)| \fa \in
	\Phi \} \leq 0,$$ thus $\Supp_R(\lc^{i}_\Phi(M))\subseteq \Max(R)$. Now, it follows by Corollary \ref{3.4} that $\lc^{i}_\Phi(M)$ is an Artinian $R$-module for all $i\geq 0$. Therefore \cite[Corollary 2.10]{AB1} implies that $\lc^{i}_\Phi(M)$ is $\Phi$-cofinite for all $i\geq 0$.
\end{proof}


 The following example shows that  in Corollary \ref{3.4} the Artinianness of$~\lc_{\fa}^{i}(M)$  for all $\fa \in \Phi$ and
 $i < t$ for a non-negative integer $t$ or  for all $i \geq 0$, may not be  equivalent to the Artinianness of general local cohomology module $\lc_{\Phi}^{i}(M)$ for all 
 $i < t$ or  for all $i \geq 0$.

 \begin{exam}\label{ex} Let $R$  be a Gorenstein ring of finite dimension $d$ such that has infinite maximal ideal with
 	$\Ht \fm = d$. Let 
 	\begin{center}
 	$\Psi = \{\fm \in \Max(R) \vert \Ht \fm = d  \}$ and  $\Phi=\{\fa | \fa \text{ is an ideal of } R    \text{ and } \dim{R/{\fa}}\le 0 \}.$
 	\end{center}

 	Then, it is easy to see that $\Phi$ is a system of ideals of $R$ and $\Psi\subseteq \Phi$. By \cite[Theorem 18.8]{Mat}, the minimal injective resolution of $R$ has the form 
 	
 	\begin{center}
 		$0\lo\underset{\Ht_R{\fp}=0}\bigoplus\E(R/\fp)\lo\underset{\Ht_R{\fp}=1}\bigoplus\E(R/\fp)\lo\dots\lo\underset{\Ht_R{\fp}=d}\bigoplus\E(R/\fp)\lo 0$
 	\end{center}
 	
 Applying the functor $\G_{\Phi}(-)$  to the above injective resolution,  we can deduce that	
 	 
 	\begin{equation*}
 	\lc^{i}_{\Phi}(R)= \left\{
 	\begin{array}{cc}
 	\underset{\fm \in \Psi}\bigoplus E(R/ \fm)   & \text{if}     ~~~~    i = d,\\
 	0                             & \text{if}       ~~~~  i \neq d.
 	\end{array}
 	\right.
 	\end{equation*}
 	Clearly $\lc_{\Phi}^{d}(R)$ is not Artinian while it is easy to see that $\lc_{\fa}^{i}(R)$ is Artinian for all $\fa \in \Phi$
 	and all $i \geq 0$. 
 \end{exam}

The following corollary shows unlike the Artinianness, vanishing of general local cohomology modules in initial points is equivalent to the vanishing of its ordinary local cohomology.

\begin{cor}\label{cor2.15}
	Let $M$ be an arbitrary $R$-module, $\Phi$ be a system of ideals of $R$ and $t$ be a non-negative integer.  Then the following statements are equivalent:
	\begin{itemize}
		\item[(i)]  $\lc^{i}_\Phi(M)=0$ for all $i<t$ {\rm(}for all $i \geq 0${\rm)};
		\item[(ii)]  $\lc^{i}_\fa(M)=0$ for all $i<t$ {\rm(}for all $i \geq 0${\rm)}  and all  $\fa \in \Phi$.
	\end{itemize}
\end{cor}

\begin{proof}
	Take $k=-1$	in Corollary \ref{2.14}.
\end{proof}


The following corollary shows that union of supports of general local cohomology modules at initial points $i<t$ is equal to union of supports of its all ordinary local cohomology modules at initial points $i<t$ while for each non-negative integer $i$,  Since $\lc^{i}_\Phi(M)\cong\underset { \fa \in \Phi}
\varinjlim \lc^{i}_\fa(M),$ by \cite[Lemma 2.1]{BZ1}, it is easy to see that $\Supp_R(\lc^{i}_\Phi(M))
\subseteq\underset{\fa \in \Phi}\bigcup \Supp_R(\lc^{i}_\fa(M))$.

\begin{cor} \label {2.16}
	Let $M$ be an be an arbitrary $R$--module and $t$ be a non-negative integer. Then
	\begin{center}
		$\underset{i<t}\bigcup \Supp_R(\lc^{i}_{\Phi}(M))= \underset{\fa\in\Phi,i<t}\bigcup \Supp_R(\lc^{i}_{\fa}(M)).$
	\end{center}
\end{cor}

\begin{proof}
	By Corollary \ref {cor2.15}, we have\\
	\centerline{$\begin{matrix}
		\fp\notin \underset{i<t}\bigcup \Supp_R(\lc^{i}_{\Phi}(M))
		&\Leftrightarrow &\forall i< t;&\lc^{i}_{\Phi}(M)_{\fp}= 0\\
		&\Leftrightarrow &\forall i< t;&\lc^{i}_{{\Phi}_{\fp}}(M_{\fp})= 0\\
		&\Leftrightarrow &\forall{\fa}R_{\fp}\in\Phi_{\fp}, i< t;&\lc^{i}_{{\fa}R_{\fp}}(M_{\fp})= 0\\
		&\Leftrightarrow &\forall\fa\in\Phi, i<t;&\lc^{i}_{\fa}(M)_{\fp}= 0\\
		&\Leftrightarrow &&\fp\notin \underset{\fa\in\Phi, i<t}\bigcup \Supp_R(\lc^{i}_{\fa}(M)),
		\end{matrix}$}\\
	as we desired.
\end{proof}

\section{Some applications and identities}


As applicatios of  Theorems \ref{2.1} and \ref{2.8}, we can state  Corollaries \ref{4.1}, \ref{4.2} and \ref{4.3}. 

The following corollary is a generalization of \cite[Theorem 4.1(c)]{AM}.

\begin{cor}\label{4.1}
	Let $N$ be a $\Phi$-torsion and finitely generated $R$-module. Let $M$ be an arbitrary $R$-module and $t$ a non-negative integer such that $\Ext^{j-i}_R(N,\lc^{i}_\Phi(M))$ is in $\mathcal{S}$  for all $i,j$ with $0\leq i\leq t-1$ and  $j =t,t+1$. Then
	$\Ext^{t}_R(N,M)$ is in $ \mathcal{S}$ if and only if  $\Hom_R(N,\lc^{t}_\Phi(M))$ is in $\mathcal{S}$.
\end{cor}

\begin{proof}
	Follows by  Theorems \ref{2.1} and \ref{2.11}.
\end{proof}

\begin{cor}\label{4.2}
	Let $N$ be a $\Phi$-torsion and finitely generated $R$-module. Let $M$ be an arbitrary $R$-module and $s,t$ non-negative integers such that $\lc^{i}_\Phi(M)$ is in $\mathcal{S}$ for all $i$ with $0\leq i\leq t-1$ or $t+1\leq i\leq s+t$. Then  $\Ext^{s+t}_R(N,M)$  is in $\mathcal{S}$ if and only if 
	$\Ext_{R}^{s}(N, \lc^{t}_\Phi(M))  \in \mathcal{S}$.
\end{cor}

\begin{proof}
	Follows by  Theorems \ref{2.1} and \ref{2.8}.
\end{proof}


The following Corollary is a generalization of \cite[Theorem 2.8]{ATV}.

\begin{cor}\label{4.3}
	Let $M$ be an $R$-module and $s,t$  be non-negative integers such that $t \leq s$. Assume also that $\lc^{i}_\Phi(M)$ is in $\mathcal{S}$,  for all  $i$, $i \neq t$ \rm(respectively  $0 \leq i \leq t-1$ or $t+1 \leq i \leq s $\rm). Then for all $i$, $i \geq t$ \rm ( respectively $0 \leq i \leq t-s$ \rm)  $\Ext_{R}^{i}(N , \lc^{t}_\Phi(M)) $  is in $\mathcal{S}$ if and only if 
	$\Ext_{R}^{i+t}(N , M)  \in \mathcal{S} $.
\end{cor}

\begin{proof}
Easily shown with  Theorems \ref{2.1} and \ref{2.8}.
\end{proof}


In the following theorem, for non-negative integers $s$ and $t$, we find some
sufficient conditions for validity of the isomorphism  $\Ext^{s+t}_R(N,M)\cong \Ext^{s}_R(N ,\lc^{t}_\Phi(M))$. It is a generalization of \cite[Theorem 3.5]{ATV}.

\begin{thm}\label{zero}
Let $N$ be a $\Phi$-torsion and finitely generated $R$-module. Let  $M$ be an arbitrary  $R$-module and $s , t$  be non-negative integers. Assume also as follows:
\begin{itemize}
\item[(i)] $\Ext^{s+t-i}_R(N ,\lc^{i}_\Phi(M)) = 0$ for all $i , 0 \leq i < t$ or $t+1 \leq i \leq s+t$;
\item[(ii)] $\Ext^{s+1+i}_R(N ,\lc^{t-i}_\Phi(M)) = 0$ for all $i , 0 \leq i \leq t$;
\item[(iii)] $\Ext^{s-1-i}_R(N ,\lc^{t+i}_\Phi(M)) = 0$ for all $i , 0 \leq i \leq s-1$.
\end{itemize}
Then we have  $\Ext^{s+t}_R(N,M))\cong \Ext^{s}_R(N ,\lc^{t}_\Phi(M))$.
\end{thm}

\begin{proof}
We prove by using induction on $t$. Let $t = 0$. We have
\begin{center}
 $\Ext^{s-1}_R(N , M / \Gamma_{\Phi}(M)) = 0 = 
\Ext^{s}_R(N , M / \Gamma_{\Phi}(M))$
\end{center}
  from hypothesis $(iii)$ and $(i)$, and Theorem \ref{2.1} with $\mathcal{S}=\{0\}$. Now, the assertion follows by the exact sequence

\begin{center}
$\Ext_{R}^{s-1}(N, M/\Gamma_{\Phi}(M))\lo \Ext^{s}_{R}(N, \Gamma_{\Phi}(M))\lo \Ext^{s}_{R}(N, M)\lo \Ext^{s}_{R}(N, M/\Gamma_{\Phi}(M))$
\end{center}

obtained from the short exact sequence
\begin{center}
$0\lo \Gamma_{\Phi}(M)\lo M\lo M/\Gamma_{\Phi}(M)\lo 0$.
\end{center}

Assume that $t> 0$ and that $t- 1$ is settled. By considering the short exact sequence
\begin{center}
	$0\rightarrow \overline{M}\rightarrow E(\overline{M})\rightarrow L\rightarrow 0,$
\end{center}	
the rest of proof is similar to that of Theorem \ref {2.1}. 
\end{proof}




The following corollary is a generalization of \cite[Corollary 3.8]{ATV}.

\begin{cor}
Let $M$ be an arbitrary $R$-module, $\fa$ an deal of $R$ and $\fp \in \bigcup_{\fa \in \Phi} \V(\fa)$. Suppose that $\Ext^{s+t-i}_R(R / \fp,\lc^{i}_\Phi(M)) =0$ for all $i$ with $0 \leq i < t$ or $t+1 \leq i \leq s+t$, 
$\Ext^{s++1+i}_R(R / \fp,\lc^{i}_\Phi(M)) =0$ for all  $i , 0 \leq i \leq t$, and   $\Ext^{s-1-i}_R(R / \fp,\lc^{i}_\Phi(M)) =0$ for all  $i , 0 \leq i \leq s-1$. Then $\mu^{s+t}(\fp , M) = \mu^{s}(\fp , \lc^{t}_\Phi(M)) $.
\end{cor}
\begin{proof}
Note that by Theorem \ref{zero}, $\Ext^{s+t}_{R_{\fp}}(R_{\fp} / \fp R_{\fp}, M_{\fp})\cong \Ext^{s}_{R_{\fp}}(R_{\fp} / \fp R_{\fp} ,
\lc^{t}_\Phi(M)_{\fp})$ . Therefore  $\mu^{s+t}(\fp , M) = \mu^{s}(\fp , \lc^{t}_\Phi(M))$. 

\end{proof}

\begin{cor}\label{3.3}
Let $M$ be an arbitrary $R$-module and $N$ be a finitely generated  $\fa$-torsion $R$-module for some $\fa \in \Phi $. Suppose that $\Ext^{s+t-i}_R(N,\lc^{i}_\Phi(M)) =0$ for all $i , 0 \leq i < t$ or $t+1 \leq i \leq s+t$, 
$\Ext^{s+1+i}_R(N,\lc^{i}_\Phi(M)) =0$ for all  $i , 0 \leq i \leq t$, and   $\Ext^{s-1-i}_R(N,\lc^{i}_\Phi(M)) =0$ for all  $i , 0 \leq i \leq s-1$. Then: 
\begin{center}
$\Ext^{s}_R(N,\lc^{t}_\Phi(M)) \cong \Ext^{s+t}_R(N,M) \cong \lc^{s+t}_\fa(N,M)  \cong \Ext^{s}_R(N,\lc^{t}_\fa(M))$
\end{center}
\end{cor}

\begin{proof}
It follows by Theorem \ref{zero} and \cite[Lemma 2.5(c) and Theorem 2.21]{VA}. Note that $N = \Gamma_\fa (N) \subseteq \Gamma_\Phi(N)$.
\end{proof}


\begin{cor}\label{4.7}
 Let $M$ be an arbitrary $R$-module and $N$ be a finitely generated   $R$-module with $\Supp_R(N) = \V(\fa)$  for some $\fa \in \Phi $.  If $\lc^{i}_\Phi(M) = 0$ for all $0 \leq i < t$,  then:
\begin{center}
	 $\Hom_R(N,\lc^{t}_\Phi(M)) \cong  \Ext^{t}_R(N,M) \cong  \lc^{t}_\fa(N,M) \cong \Hom_R(N,\lc^{t}_\fa(M)).$
\end{center} 
     
\end{cor}
\begin{proof}
Put $s=0$ in Corollary \ref{3.3}.
\end{proof}


\begin{defn}\textsl{{\rm(}\cite[Definition 5.3]{BZ1}{\rm)}}
Let $M$ be a finitely generated $R$-module and $\Phi$ a system of ideals of $R$. We define  grade of $\Phi$ on $M$ as:
\begin{center}
 $\grade(\Phi,M)=\inf\{\grade(\fa,M)| \fa\in \Phi\}$.
\end{center}
\end{defn}
It is easy to see that, if for each $\fa\in \Phi$, $M= {\fa}M$ then, $\grade(\Phi,M)=+\infty$, otherwise we have $\grade(\Phi,M)<+\infty$.

\begin{cor}\label{4.9}
	Let $M$, $N$ be two finitely generated   $R$-modules such that $\Supp_R(N) = \V(\fa)$  for some $\fa \in \Phi $.  If $\grade(\Phi,M)=t$, then 
	\begin{center}
		$\Hom_R(N,\lc^{t}_\Phi(M)) \cong  \Ext^{t}_R(N,M) \cong  \lc^{t}_\fa(N,M) \cong \Hom_R(N,\lc^{t}_\fa(M)).$
	\end{center} 
	\end{cor}
	
\begin{proof}
Let $\fa \in \Phi$, then $\grade(\fa,M)> t$. Therefore $\lc^{i}_\fa(M)=0$ for all $i<t$. Now the assertion follows by Corollaries \ref{cor2.15} and \ref{4.7}.
\end{proof}

 A subset $\mathcal{Z}$ of $\Spec(R)$ is said to be {\it stable under specialization} if $\V(\fp)\subseteq \mathcal{Z}$ for all $\fp\in \mathcal{Z}$. Let $M$ be an $R$-module, then $\Gamma_{\mathcal{Z}}(M)$ is defined by
 $\Gamma_{\mathcal{Z}}(M):=\{~x\in M\mid \Supp_R(Rx)\subseteq \mathcal{Z}\}$. So $\G_{\Phi}(-)$ is a functor from $\mathscr{C}(R)$ to itself (where
 $\mathscr{C}(R)$ denotes the category of all $R$-modules and all
 $R$-homomorphisms). The functor  $\G_{\mathcal{Z}}(-)$ is additive, covariant,
 $R$-linear and left exact. For
 each $i\geq 0$, the $i$-th right derived functor of $\G_{\mathcal{Z}}(-)$ is denoted by
 $\lc_{\mathcal{Z}}^i(-)$. By \cite[Lemma 3.2]{DNT}, for any stable under specialization subset $\mathcal{Z}$ of $\Spec(R)$, there is a system of ideals $\Phi$ such that for each $R$-module $M$, the $R$-modules $\lc_{\mathcal Z}^{i}(M)$ and $ \lc^{i}_{\Phi}(M)$ are isomorphic. 
 
Let $(R,\fm)$ be a local ring and $\fa$ be an ideal of $R$. Let $M$ be a finitely generated $R$-module with $\dim M=d$. It is well-known that the set $\Supp_R(\lc^{d-1}_\fa(M))$ is finite. We close this paper with the following example. It shows that the similar result is not true for the local cohomology modules with respect to a system of ideals or with respect to an specialization closed subset.

\begin{exam}\label{4.8}
Let $(R, \fm)$  be a Noetherian local domain such that $\dim(R) = 2$. Since $\{0\}$   is a prime ideal of $R$ and $\{0\} \subset \fm$,
so there is infinitely many prime ideals of $R$ between $\{0\}$ and $\fm$ such that $\Ht_R {\fp}=1$.  Let $\mathcal Z =\Spec(R)$ and  $\mathcal Z_n = \lbrace \fp\in \mathcal Z  \mid \Ht_R {\fp}\geq n \rbrace$ for $n=0, 1, 2$. Then $\mathcal Z_2 = \{\fm\}$ and $\mathcal Z_1 =\{ \fp\in \mathcal Z  \mid \Ht_R {\fp}=1\} \cup\{\fm\}$. According to \cite[Corollary 4.3 and Lemma 4.2]{DZ}, for $n=2$, $i=1$ and $X=R$, there is the following exact sequence 

\begin{center}
$\lc^{1}_{\fm}(R) \lo \lc^{1}_{\mathcal Z_1}(R) \lo \underset{{\Ht_R {\fp}=1}}\bigoplus\lc^{1}_{{\fp}R_{\fp}}(R_{\fp})
\lo \lc^{2}_{\fm}(R)$.
\end{center}
By localizing this exact sequence at every $\fp$ with $\Ht_R {\fp}=1$, we get $\lc^{1}_{\fm}(R)_\fp=0=\lc^{2}_{\fm}(R)_\fp$ and $\lc^{1}_{\mathcal Z_1}(R)_\fp\cong \lc^{1}_{{\fp}R_{\fp}}(R_{\fp})$. Since $\lc^{1}_{{\fp}R_{\fp}}(R_{\fp})\neq 0$ by Grothendieck non-vanishing theorem \cite[Theorem 6.1.4]{BSh2}, thus $\fp \in\Supp_R(\lc^{1}_{\mathcal Z_1}(R))$. Also as we mentioned before there is a system of ideals $\Phi$ such that the $R$-modules $\lc_{\mathcal Z_1}^{1}(R)$ and $ \lc^{1}_{\Phi}(R)$ are isomorphic. This shows that $\Supp_R(\lc^{1}_{\mathcal Z_1}(R)) =\Supp_R(\lc^{1}_\Phi(R))$ is an infinite set.
\end{exam}


\begin{center}
{\bf Acknowledgments}
\end{center}

The authors would like to thank Dr. Majid Rahro Zargar for his reading of the first draft and useful discussions.



\bibliographystyle{amsplain}

\end{document}